\pdfoutput=1
\RequirePackage{ifpdf}
\ifpdf 
\documentclass[pdftex]{sigma}
\else
\documentclass{sigma}
\fi

\numberwithin{equation}{section}

\newtheorem{Theorem}{Theorem}[section]
\newtheorem{Corollary}[Theorem]{Corollary}
\newtheorem{Lemma}[Theorem]{Lemma}
\newtheorem{Proposition}[Theorem]{Proposition}
 { \theoremstyle{definition}
\newtheorem{Example}[Theorem]{Example}
\newtheorem{Remark}[Theorem]{Remark} }

\usepackage[all]{xy}
\usepackage{slashed}
\usepackage{tikz-cd}
\usetikzlibrary{matrix,arrows,decorations.pathmorphing}

\def\CC{\ensuremath{\mathbb{C}}}
\def\QQ{\ensuremath{\mathbb{Q}}}
\def\RR{\ensuremath{\mathbb{R}}}
\def\ZZ{\ensuremath{\mathbb{Z}}}
\newcommand{\integers}{\mathbb{Z}}
\newcommand{\rationals}{\mathbb{Q}}
\DeclareMathOperator{\Ind}{Ind}
\newcommand{\tensor}{\otimes}
\newcommand{\into}{\hookrightarrow}
\newcommand{\iso}{\cong}
\newcommand{\disjointunion}{\amalg}
\DeclareMathOperator{\Pos}{Pos}
\DeclareMathOperator{\im}{im}
\DeclareMathOperator{\spin}{spin}
\DeclareMathOperator{\coker}{coker}

\begin{document}
\allowdisplaybreaks

\newcommand{\arXivNumber}{2006.15965}

\renewcommand{\thefootnote}{}

\renewcommand{\PaperNumber}{129}

\FirstPageHeading

\ShortArticleName{Positive Scalar Curvature due to the Cokernel of the Classifying Map}

\ArticleName{Positive Scalar Curvature due to the Cokernel\\ of the Classifying Map\footnote{This paper is a~contribution to the Special Issue on Scalar and Ricci Curvature in honor of Misha Gromov on his 75th Birthday. The full collection is available at \href{https://www.emis.de/journals/SIGMA/Gromov.html}{https://www.emis.de/journals/SIGMA/Gromov.html}}}

\Author{Thomas SCHICK~$^\dag$ and Vito Felice ZENOBI~$^\ddag$}

\AuthorNameForHeading{T.~Schick and V.F.~Zenobi}

\Address{$^\dag$~Mathematisches Institut, Universit\"at G\"ottingen, Germany}
\EmailD{\href{mailto:thomas.schick@math.uni-goettingen.de}{thomas.schick@math.uni-goettingen.de}}
\URLaddressD{\url{http://www.uni-math.gwdg.de/schick/}}

\Address{$^\ddag$~Dipartimento di Matematica, Sapienza Universit\`{a} di Roma,\\
\hphantom{$^\ddag$}~Piazzale Aldo Moro 5 - 00185 - Roma, Italy}
\EmailD{\href{mailto:vitofelice.zenobi@uniroma1.it}{vitofelice.zenobi@uniroma1.it}}

\ArticleDates{Received July 13, 2020, in final form December 04, 2020; Published online December 09, 2020}

\Abstract{This paper contributes to the classification of positive scalar curvature metrics up to bordism and up to concordance. Let $M$ be a closed spin manifold of dimension $\ge 5$ which admits a metric with positive scalar curvature. We give lower bounds on the rank of the group of psc metrics over $M$ up to bordism in terms of the corank of the canonical map $KO_*(M)\to KO_*(B\pi_1(M))$, provided the rational analytic Novikov conjecture is true for~$\pi_1(M)$.}


\Classification{53C20; 53C21; 53C27; 55N22; 19K56; 19L64}

\renewcommand{\thefootnote}{\arabic{footnote}}
\setcounter{footnote}{0}

\section{Introduction}

 The study of metrics with positive scalar curvature is nowadays the focus of a
 very active area of research. The starting point typically will be a closed
 spin manifold $M$, and one would like to get suitable information about the
 possible Riemannian metrics on $M$.

 Stephan Stolz introduced a long exact sequence
 for the systematic bordism classification of metrics of positive scalar
 curvature. For this, one has to fix an additional reference space $X$. Then
 this sequence is given by (ending with $n=5$ at the right)
 \begin{equation}\label{ses}
 \xymatrix{\cdots\ar[r]&\mathrm{R}^{\spin}_{n+1}(\Gamma)\ar[r]^\partial&\mathrm{Pos}^{\spin}_{n}(X)\ar[r]& \Omega^{\spin}_n(X)\ar[r]&\mathrm{R}^{\spin}_{n}(\Gamma)\ar[r]^\partial&\cdots.}
 \end{equation}
Here $\Omega^{\spin}_n(X)$ is the usual spin cobordism group, it consists of cobordism classes of cycles
 $f\colon M\to X$, with $M$ a closed $n$-dimensional spin manifold; $\mathrm{Pos}^{\spin}_{n}(X)$ is the group of bordism classes of metrics of
 positive scalar curvature on $n$-dimensional closed spin manifolds with reference
 map to $X$; finally, $\mathrm{R}^{\spin}_{n}(\Gamma):=\mathrm{R}^{\spin}_{n}(X)$ is a relative group
 discussed
 in more detail below, known to depend only on
 $\Gamma:=\pi_1(X)$. The group structure in each of the three cases is given by
 disjoint union.

Because the starting point typically is a fixed manifold $M$, one has to make
 a suitable choice of $X$. The standard choice here is $X=B\Gamma$ with
 $\Gamma=\pi_1(M)$. Note that with $X=B\Gamma$ the Stolz sequence then contains
 in $\mathrm{Pos}^{\spin}_{n}(B\Gamma)$ information for all spin manifolds with fundamental group~$\Gamma$ at
 once. This is the situation discussed in the majority of all the previous work.

In the current article we change the paradigm a bit. We
 argue that, starting with $M$, the choice of $X=M$ is even more canonical, and
 we study $\mathrm{Pos}^{\spin}_{n}(M)$. The
 usual applications to concordance classes of metrics of positive scalar
 curvature on $M$ can still be made, and the theory is richer and more specific.

 A very fruitful way to get information about the Stolz sequence (for
 arbitrary $X$) uses the index theory of the spin Dirac operator. A
 systematic approach was given in \cite{PiazzaSchick_Stolz}, where the authors
 construct a mapping of~\eqref{ses} to the analytic surgery exact sequence of
 Higson and Roe (where $\Gamma=\pi_1(X)$)
 \begin{equation}\label{HRses}
 \begin{split}&
 \xymatrix{\cdots\ar[r]&\mathrm{R}^{\spin}_{n+1}(\Gamma)\ar[r]^\partial\ar[d]^{{\Ind}^{\Gamma}}&\mathrm{Pos}^{\spin}_{n}(X)\ar[r]\ar[d]^{\varrho}& \Omega^{\spin}_n(X)\ar[r]\ar[d]^\beta&\mathrm{R}^{\spin}_{n}(\Gamma)\ar[r]^\partial\ar[d]^{{\Ind}^{\Gamma}}&\cdots\\
 \cdots\ar[r]^{\mu^\Gamma_X\qquad}&K_{n+1}(C^*\Gamma)\ar[r]&\mathcal{S}_n^\Gamma(\widetilde{X})\ar[r]& K_n(X)\ar[r]^{\mu^\Gamma_X}&K_n(C^*\Gamma)\ar[r]&\cdots.}
 \end{split}
 \end{equation}

A successful strategy for detecting non-trivial elements in
 $\mathrm{Pos}^{\spin}_{n}(X)$ goes as follows: if one can construct a cycle
 $\xi$ for $\mathrm{R}^{\spin}_{n+1}(\Gamma)$ such that ${\Ind}^{\Gamma}(\xi)$
 maps to a nonzero element in the cokernel of $\mu_X^\Gamma$, then $\partial(\xi)$ is non-zero in
 $\mathrm{Pos}^{\spin}_{n}(X)$ because its image through $\varrho$ does not
 vanish.

 Indeed, this is the line followed by Weinberger and Yu in
 \cite{WeinbergerYu}, where the authors define the so-called \emph{finite part}
 of the K-theory of the maximal group C*-algebra, which is proven to map
 injectively to the
 cokernel of the assembly map. Along with this they give the concrete
 construction of elements in $\mathrm{R}^{\spin}_{n+1}(\Gamma)$ whose higher
 index belongs to this finite part. Xie, Yu and Zeidler in~\cite{XieYuZeidler} have
 systematized those constructions and corrected some mistakes, giving a more exhaustive description of the images of the vertical arrows in~\eqref{HRses}.
 These are complemented by a long line of results which instead make use of
 higher numerical invariants, such as~\cite{BotvinnikGilkey_eta,LeichtnamPiazza}, where higher $\eta$-invariants are used, or~\cite{PiazzaSchick}, where Cheeger--Gromov $L^2$-$\varrho$-invariants
 play an important role.

 Another example of how K-theory methods could improve those which use
 numerical invariants can be appreciated by comparing
 \cite{PiazzaSchickZenobi_psc} with \cite{KazarasRubermanSaveliev}, where
 $\eta$-invariants on end-periodic ends are used in order to study positive
 scalar curvature metrics on even dimensional manifolds. We don't want
 to repeat the results of this work in detail. The general pattern is: the
 invariants mentioned are constructed and shown to be invariants of classes
 in $\mathrm{Pos}_*^{\spin}(X)$ or of the concordance classes, and then (many)
 elements are constructed which are distinguished by these invariants,
 giving rise to interesting lower bounds on the rank of
 $\mathrm{Pos}^{\spin}_*(X)$.

 All the work described so far uses fundamentally that
 the group
 $\Gamma$ contains non-trivial torsion. In particular, as it is explicitly
 explained in
 \cite{XieYuZeidler}, the ultimate source of those constructions is the
 difference between $\underline{E}\Gamma$, the classifying space for proper
 actions, and $E\Gamma$, the classifying space for proper and free
 actions.

 Now, if $\Gamma$ is torsion-free, then $\underline{E}\Gamma$ and $E\Gamma$
 coincide. Therefore one has to find an alternative source for non-trivial
 elements in $\mathrm{R}^{\spin}_{n+1}(\Gamma)$ whose higher index maps to
 non-trivial elements of the cokernel of the assembly map.

 The main method of this paper is to use the homological difference between
 $M$ and $B\Gamma$ for this purpose. More generally, given $X$ (which could
 be $M$) with a classifying map $u\colon X\to B\Gamma$ such that the homological
 difference between them is rich, we can construct non-trivial elements in
 $\mathrm{R}^{\spin}_{n+1}(\Gamma)$. Although our main motivation was to
 obtain results for torsion-free fundamental groups, our constructions
 work for arbitrary $\Gamma$. In particular, we prove the following
 result.

 \begin{Theorem}\label{lower-bound} Let $X$ be a finite connected CW-complex
 and $n\ge 5$. Let $u\colon X\to B\pi_1(X)$ be the classifying map of its universal covering.
 Let us assume that the $($rational$)$ strong Novikov conjecture holds for
 $\Gamma:=\pi_1(X)$, i.e., the assembly map $K_*(B\Gamma)\otimes\QQ\to
 K_*(C^*\Gamma)\otimes\QQ$ is injective. Set
 \[k:= \dim\bigg(\mathrm{coker}\bigg(\bigoplus_{j\geq0}H_{n+1-4j}(X;\QQ)\xrightarrow{u_*}\bigoplus_{j\geq0}H_{n+1-4j}(B\Gamma;\QQ)\bigg)\bigg), \]
 and
 \[k':= \dim\bigg(\mathrm{ker}\bigg(\bigoplus_{j\geq0}H_{n-4j}(X;\QQ)\xrightarrow{u_*}\bigoplus_{j\geq0}H_{n-4j}(B\Gamma;\QQ)\bigg)\bigg),\]
 then
 \[\operatorname{rk} \mathrm{Pos}^{\spin}_{n}(X)\geq k+k'.\]
 \end{Theorem}

 The reason why in the statement of Theorem~\ref{lower-bound} we
 distinguish the elements coming from the kernel and those from the cokernel
 of $u_*$ is that, if $X$ is an $n$-dimensional closed spin manifold, those coming
 from the cokernel can be realized by cycles in $ \mathrm{Pos}^{\spin}_{n}(X)$
 represented by the identity map. This fact is key for the main result of
 Section~\ref{4}.

By standard surgery techniques we can refine the previous result if we look
 for metrics on a fixed manifold $M$. To formulate this, denote by $P^+(M)$ the set of concordance classes of metrics with positive scalar curvature on an $n$-dimensional closed spin manifold $M$.
 In the proof of \cite[Theorem~5.4]{Stolz}, in order to construct a free and
 transitive action of $\mathrm{R}^{\spin}_{n+1}(\Gamma)$ on $P^+(M)$, Stolz
 defines a ``difference'' map
 \begin{equation}\label{i-map}i\colon \ P^+(M)\times P^+(M)\to \mathrm{R}^{\spin}_{n+1}(\Gamma)\end{equation} such that
 \begin{itemize}\itemsep=0pt
 \item $i(g,g)=0$ and $i(g,g')+i(g',g'')= i(g,g'')$ for all $g,g',g''\in P^+(M)$;
 \item the map $i_g\colon P^+(M)\to \mathrm{R}^{\spin}_{n+1}(\Gamma)$, which sends $g'$ to $i(g,g')$ is bijective for all $g\in P^+(M)$.
 \end{itemize}
 This induces on $P^+(M)$ the structure of an
 $\mathrm{R}^{\spin}_{n+1}(\Gamma)$-torsor, or the structure of an affine space
 modelled on $\mathrm{R}^{\spin}_{n+1}(\Gamma)$.
 After picking any point $g_0$ of
 $P^+(M)$
 as the identity,
 $P^+(M)$
 acquires a~group structure isomorphic to
 $\mathrm{R}^{\spin}_{n+1}(\Gamma)$. This group structure is
 non-canonical as it depends on~$g_0$ and therefore seems only useful if there is a preferred~$g_0$
 (e.g., one which bounds a metric of positive scalar curvature, as the standard
 metric on~$S^n$). This kind of structure is studied (and improved to an
 H-space structure on the space of metrics of positive scalar
 curvature) in~\cite{Frenck}. We use the affine structure of~$P^+(M)$
 in the last part of the following Theorem~\ref{rank>k}.

 \begin{Theorem}\label{rank>k}
 Let $X$, $n$, $k$ and $\Gamma$ be as in Theorem~{\rm \ref{lower-bound}}.
 Assume that there exists a cycle in~$\Pos^{\spin}_n(X)$, given by $(f\colon M\to X, g)$ such that $f$ is $2$-connected $($i.e., inducing an isomorphism on $\pi_0$ and $\pi_1$ and a surjection on~$\pi_2)$.

 Then there are metrics with positive scalar curvature $g, g_1,\dots,g_k$ on~$M$, together with the
 fixed map $f\colon M\to X$, which
 \begin{enumerate}\itemsep=0pt
 \item[$(1)$] span an affine lattice of rank $k$ in
 the abelian group $\Pos^{\spin}_n(X)$ and hence an affine space of dimension $k$ in
 $\Pos^{\spin}_n(X)\otimes\rationals$;
 \item[$(2)$] in particular, they span an affine lattice of rank $k$
 in $\Pos^{\spin}_n(M)$ $($with reference map the identity$)$;
 \item[$(3)$] in particular, they span an affine lattice of rank~$k$ of concordance classes of positive
 scalar curvature metrics in~$P^+(M)$. The lattice is
 with respect to the underlying
 structure of an affine space modelled on the abelian group~$R^{\spin}_{n+1}(\Gamma)$.
 \end{enumerate}
 \end{Theorem}

 Perhaps the first result which uses index methods to classify metrics of
 positive scalar curvature is obtained by Carr \cite{Carr}, where infinitely many concordance
 classes of metrics with positive scalar curvature are constructed even on
 simply
 connected manifolds $M$ like the sphere (of the right dimension). This is
 different in spirit to
 our result: we prove in Remark \ref{carr} that the classes of Carr are all equal in $
 \Pos^{\spin}_n(M)$, i.e., although they are not concordant, they are all
 bordant.

 Recently, Ebert and Randal-Williams in \cite{EbertRW} developed a
 very sophisticated bordism category approach to study $\mathcal{R}^+(M)$,
 the space of the metrics with positive
 scalar curvature on $M$. Theorem~C of~\cite{EbertRW} implies that, if $M$ has even dimension $2n$, the fundamental
 group $\Gamma$ verifies rationally the Baum--Connes conjecture and its
 homological dimension is less or equal to $2n+1$, then the so-called index
 difference map
 is a rational surjection of $\pi_0(\mathcal{R}^+(M))$ onto~$KO_{2n+1}(C^*\Gamma)$.

 In our results, we only assume the
 rational injectivity instead of bijectivity of the Baum--Connes assembly map
 for $\Gamma$ and, as remarked in contrast with Carr, see \cite{Carr}, we obtain metrics
 which are not only non-isotopic, but also non-bordant. On the other hand, in
 \cite{EbertRW} the authors are mainly interested in higher homotopy
 groups.

 Finally, we provide a detailed and pedestrian proof how to pass from a bordism $W\xrightarrow{F}X$ between
 $M_0\xrightarrow{f_0}X$ and $M_1\xrightarrow{f_1}X$ to a bordism
 $W'\xrightarrow{F'}X$ with same ends which we call \emph{Gromov--Lawson
 admissible}, meaning that it is built from $M_0$ by attaching
 handles of codimension $\ge 3$, provided that
 $f_1$ is 2-connected. This is certainly a well known and heavily used
 result, but does not seem treated well in a pedestrian way with all details,
 which we try to provide here.

 The paper is organized as follows:
 \begin{itemize}\itemsep=0pt
\item In Section \ref{3} we prove Theorem \ref{lower-bound}, which gives a lower bound for the rank of $\mathrm{Pos}^{\spin}_{n}(X)$ in term of the difference between $X$ and $B\Gamma$.
 \item In Section \ref{4} we prove Theorem \ref{rank>k}, which refines
 Theorem \ref{lower-bound} to a result about concordance classes. In
 particular we give details how bordisms can
 be made Gromov--Lawson admissible in the sense mentioned above.
 \end{itemize}

\section{Mapping psc to analysis to detect bordism classes}\label{3}
In \cite[Section 5]{PiazzaSchick_Stolz} Piazza and Schick construct a map from
the Stolz exact sequence to the Higson--Roe exact sequence (see also
\cite{XieYu,ZenobiAdiabatic} for different approaches). Instead of working
with complex $C^*$-algebras as in \cite{PiazzaSchick_Stolz}, one can without
extra effort adapt this construction to the setting of real $C^*$-algebras (compare
\cite{Zeidler}). All of the constructions are natural. As a result,
for a finite
connected CW-complex $X$ with $\Gamma=\pi_1(X)$ and classifying map $u\colon X\to
B\Gamma$ for its universal covering we obtain
the following commuting diagram of Stolz exact sequences
\begin{equation}\label{diagramStolz}\begin{split}&
 \xymatrix{
 \ar[r]
 &\mathrm{Pos}^{\spin}_{n+1}(X)\ar[r]\ar[d]^{u_*}& \Omega^{\spin}_{n+1}(X)\ar[r]^{j_X}\ar[d]^{{u_*^{\Omega}}}&\mathrm{R}^{\spin}_{n+1}(X)\ar[r]^{\partial_X}\ar[d]^\cong&\mathrm{Pos}^{\spin}_{n}(X)\ar[r]\ar[d]^{u_*}&\\
 \ar[r]
 &\mathrm{Pos}^{\spin}_{n+1}(B\Gamma)\ar[r]&
 \Omega^{\spin}_{n+1}(B\Gamma)\ar[r]^j&\mathrm{R}^{\spin}_{n+1}(\Gamma)\ar[r]^{\partial}& \mathrm{Pos}^{\spin}_{n}(B\Gamma)\ar[r]&,
 }\end{split}
\end{equation}
which is mapped to the corresponding diagram of Higson--Roe sequences
\begin{equation}\label{diagramHR}\begin{split}&
 \xymatrix{
 \ar[r]
 &\mathcal{S}O^\Gamma_{n+1}\big(\widetilde{X}\big)\ar[r]\ar[d]^{u_*}& KO_{n+1}(X)\ar[r]^{\mu^\Gamma_X}\ar[d]^{u_*^{KO}}&KO_{n+1}(C^*\Gamma)\ar[r]^(.6){\iota_X}\ar[d]^= &\mathcal{S}O^\Gamma_{n}\big(\widetilde{X}\big)\ar[r]\ar[d]^{u_*}&\\
 \ar[r]
 &\mathcal{S}O^\Gamma_{n+1}\ar[r]&
 KO_{n+1}(B\Gamma)\ar[r]^{\mu^\Gamma_{B\Gamma}}&KO_{n+1}(C^*\Gamma)\ar[r]^(.6){\iota}& \mathcal{S}O^\Gamma_{n}\ar[r]&.
 }\end{split}
\end{equation}
The relevant maps are the transformation of homology theories
$\beta\colon \Omega^{\spin}\to KO$, the APS-index map ${\Ind}_\Gamma\colon
R^{\spin}\to KO(C^*\Gamma)$ and the secondary index map $\rho\colon
\mathrm{Pos}^{\spin}\to \mathcal{S}O^\Gamma$.
Moreover, we set
$\mathcal{S}O^\Gamma_n\big(\widetilde{X}\big):=KO_n\big(D_{\RR}^*\big(\widetilde{X}\big)^\Gamma\big)$,
the K-theory of the Roe's $D^*$-algebra. Finally, the universal analytic structure
group $\mathcal{S}O^\Gamma_n$ is the limit of $\mathcal{S}O^\Gamma_n(Z)$ over
all $\Gamma$-compact subspaces $Z$ of $E\Gamma$.

Recall that the Pontrjagin character
$
{\rm Ph}\colon KO_*(X)\to \bigoplus_{j\in \ZZ}H_{*+4j}(X;\QQ)
$
is defined as the composition of the complexification map in K-homology
$KO_*(X)\xrightarrow{\otimes\CC}K_*(X)$ and the Chern character ${\rm Ch}\colon
K_*(X)\to \bigoplus_{k\in \ZZ}H_{*+2k}(X;\QQ)$. It so happens that ${\rm Ph}$
takes values in the subgroup $\bigoplus_{j\in\ZZ} H_{*+4j}(X;\QQ)$
and is a rational isomorphism onto that subgroup. Rationally, connective
real K-homology is a subspace of periodic real K-homology,
$ko_*(X)\tensor\rationals\subset KO_*(X)\tensor\rationals$, under the
Pontryagin character isomorphism
$ko_*(X)\tensor\rationals\xrightarrow[\iso]{{\rm Ph}}
\bigoplus_{j\ge 0} H_{*-4j}(X;\rationals)$.

\begin{Lemma}\label{pontrjagin-surjective}
 Let $X$ be a space and $n\geq0$. Then
 the composition
 \begin{equation*}
 \Omega^{\spin}_{n}(X)\otimes\QQ \xrightarrow{\beta}
 ko_{n}(X)\otimes\QQ \xrightarrow[\cong]{{\rm Ph}} \bigoplus_{j\geq0}H_{n-4j}(X;\QQ),\end{equation*}
 which assigns to a cobordism class $\big[M\xrightarrow{f}X\big]$ the class $f_*\big(\hat{A}(M)\cap [M]\big)$,
 is a natural surjection. The same holds for the relative groups
 of a map $u\colon X\to Y$.
\end{Lemma}
\begin{proof}
 If $x\in H_{n-4j}(X;\QQ)$ then by \cite[Proposition~3.1]{XieYuZeidler} there
 exists a spin manifold $M$ of dimension $n-4j$ and a map $f\colon M\to X$
 such that a non-zero multiple of $x$ is the Pontrjagin character of
 $\beta([f\colon M\to X]) =f_*[\slashed{D}_M]\in
 KO_{n-4j}(X)\otimes\QQ$. Finally, recall the Kummer surface~$V$, a spin manifold whose index generates $KO_4(*)\otimes \QQ$.
 Observe that the cartesian product of $f\colon M\to
 X$ with $V^j\to *$ is $n$-dimensional with a map to~$X$ such
 that the push-forward of its Pontrjagin character is still a non-zero
 multiple~$x$, thanks to
 the multiplicativity of the Pontrjagin character with respect to Cartesian products.

 The relative generalized homology groups of~$u$ are the (reduced) generalized
 homology groups of the mapping cone of~$u$. Therefore, the
 absolute version immediately implies the relative version.
\end{proof}

We are now able to prove the first main result of this paper.
\begin{proof}[Proof of Theorem \ref{lower-bound}]
 It is well known that the natural map $KO_*(B\Gamma)\otimes \QQ\to
 KO_*^{\Gamma}(\underline{E}\Gamma)\otimes\QQ$ is injective, compare,
 e.g., \cite[Section~7]{BCH}. Secondly,
 the rational strong Novikov conjecture for real and complex K-theory are
 equivalent, compare, e.g.,~\cite{Schick_real}. Therefore, if the strong Novikov
conjecture holds for $\Gamma$, it follows that $\mu_{B\Gamma}^\Gamma\colon
 KO_*(B\Gamma)\otimes \QQ\to KO_*(C^*\Gamma)\otimes\QQ $ is injective.

 As $\Omega^{\spin}_{n+1}(X)$ is mapped to $KO_{n+1}(X)$ under the map
 $\beta$
 from the Stolz sequences \eqref{diagramStolz} to the Higson--Roe sequences~\eqref{diagramHR} of \cite[Section~5]{PiazzaSchick_Stolz}, after factoring out their images in the appropriate places, the following maps are well-defined: $ \beta_*\colon \coker_{n+1}\big({u_*^\Omega}\big)\to \coker_{n+1}\big({u_*^{KO}}\big)$ and $({\Ind}_\Gamma)_*\colon \coker_{n+1}(j_X)\to \coker_{n+1}\big(\mu_X^\Gamma\big)$.
 Hence we obtain the following commuting diagram (using along the way the
 inverses of the third vertical arrow in~\eqref{diagramStolz} or~\eqref{diagramHR}, which are
 isomorphisms)
 \begin{equation}\label{eq:coker_diag}\begin{split}&
 \xymatrix{
 \Omega^{\spin}_{n+1}(B\Gamma) \ar@{->>}[r] \ar[d]^{\beta} &
 \coker_{n+1}\big({u_*^\Omega}\big) \ar[r]^{\quad j_*} \ar[d]^{\beta_*} &
 \coker_{n+1}(j_X) \ar@{^{(}->}[r]^{\quad(\partial_X)_*} \ar[d]^{({\Ind}_\Gamma)_*}& \mathrm{Pos}_n^{\spin}(X) \ar[d]^{\varrho}\\
 KO_{n+1}(B\Gamma) \ar@{->>}[r] & \coker_{n+1}\big({u_*^{KO}}\big) \ar[r]^{\quad(\mu^\Gamma_{B\Gamma})_*} &
 \coker_{n+1}\big(\mu_X^\Gamma\big) \ar@{^{(}->}[r]^{\quad(\iota_X)_*} &
 \mathcal{S}O^\Gamma_n\big(\widetilde{X}\big).
 }\end{split}
 \end{equation}
 The two rightmost horizontal arrows are injective by the exactness of the top
 rows of \eqref{diagramStolz} and~\eqref{diagramHR}. The middle lower arrow
 $\big(\mu^\Gamma_{B\Gamma}\big)_*$ becomes injective after tensoring with~$\mathbb{Q}$, due to the assumption that $\Gamma$ satisfies the (rational)
 strong Novikov conjecture, i.e., that $\mu^\Gamma_{B\Gamma}$ is injective
 after tensoring with~$\mathbb{Q}$, and an injective map remains
 injective if we quotient out the images of the same group (here
 $KO_{n+1}(X)\tensor\rationals$) in source and target.

By using Lemma~\ref{pontrjagin-surjective}
 and that
 \[k=\dim\bigg(\mathrm{coker}\bigg(\bigoplus_{j\geq0}H_{n+1-4j}(X;\QQ)\xrightarrow{u^H_{*}} \bigoplus_{j\geq0}H_{n+1-4j}(B\Gamma;\QQ)\bigg)\bigg),\]
 we pick $x_1,\dots,x_k\in \Omega^{\spin}_{n+1}(B\Gamma)$ such that their images
 span a $k$-dimensional subspace in
 \[\coker\bigg(\bigoplus_{j\geq0}u^H_{n+1-4j}\bigg)\cong
 \coker_{n+1}\big({u_*^{ko}}\big)\tensor\mathbb{Q} \subset
 \coker_{n+1}\big({u_*^{KO}}\big)\tensor\mathbb{Q}.\] By
 commutativity of \eqref{eq:coker_diag} and the injectivity of the lower row
 after tensoring with $\rationals$, the images of $x_1,\dots,x_k$ in
 $\Pos^{\spin}_n(X)$ under the composition of the top horizontal arrows then
 span a free abelian subgroup~$W$ of $\Pos^{\spin}_n(X)$ of rank~$k$.

 Let $\Omega^{\spin}_{*}(u)$ be the relative generalized homology group (here
 spin bordism) for the map $u\colon X\to B\Gamma$. In the following commutative diagram
 \[
 \xymatrix{\Omega^{\spin}_{n+1}(u)\tensor\rationals\ar@{->>}[r]\ar@{->>}[d]& \ker_n\big({u^\Omega_*}\big)\otimes\QQ\ar[r]\ar[d]&\Omega^{\spin}_n(X) \otimes\QQ\ar[r]^{{u^\Omega_*}}\ar[d]&\Omega^{\spin}_n(B\Gamma)\otimes\QQ\ar[d]\\
 \bigoplus\limits_{j\ge 0}
 H_{n+1-4j}(u;\rationals)\ar@{->>}[r]&\bigoplus\limits_{j\geq0}\ker_{n-4j} (u_*)\ar[r]&\bigoplus\limits_{j\geq0}H_{n-4j}(X;\QQ)\ar[r]^{u_*}&\bigoplus\limits_{j\geq0}H_{n-4j}(B\Gamma;\QQ)
 }
 \]
 the leftmost vertical arrow is surjective by the relative version of Lemma~\ref{pontrjagin-surjective} and the left horizontal arrows are surjective by
 the exactness of the pair sequence. It follows immediately that the second
 vertical arrow is also surjective.

 Now observe that $\ker_n\big({u^\Omega_*}\big)\otimes\QQ$, by the commutativity of the
 middle square in~\eqref{diagramStolz}, is also contained in the kernel of
 $j_X\colon \Omega_n^{\spin}(X)\otimes\QQ\to
 \mathrm{R}^{\spin}_n(X)\otimes\QQ$. Since
 $\dim\big(\ker_n\big({u^\Omega_*}\big)\otimes\QQ\big)\geq k'$ and by exactness of the Stolz
 sequence, it lifts to a subspace $W'$ of $\mathrm{Pos}^{\spin}_n(X)\otimes\QQ$ of
 dimension greater or equal to $k'$. It is a general fact that $W'$ is
 generated by a free abelian subgroup of~$\Pos^{\spin}_n(X)$ whose rank is
 $\dim(W')$.
 Finally, the exactness of the Stolz sequence implies that $W\tensor\rationals\cap W'=\{0\}$ in $\mathrm{Pos}^{\spin}_n(X)\otimes\QQ$ and the result is proven.
\end{proof}

\begin{Remark}
 We are thankful to a referee for stressing a more
 conceptual approach to the proof of this estimation, using more directly the
 relative generalized homology groups. By \cite[Proposition~2.4]{CSZ} and \cite[Corollary~2.6]{CSZ} we have the following long exact sequence of groups
 \begin{equation}
 \xymatrix{\cdots\ar[r]&\mathrm{Pos}^{\spin}_{n+1}(B\Gamma)\ar[r]&\Omega^{\spin}_{n+1}(u)\ar[r]&\mathrm{Pos}^{\spin}_n(X)\ar[r]^{u_*}&\mathrm{Pos}^{\spin}_n(B\Gamma)\ar[r]
 &\cdots.}\label{eq:rel_Stolz}
 \end{equation}
 Moreover in
 the following commutative diagram
 \[
 \xymatrix{\mathrm{Pos}^{\spin}_{n+1}(B\Gamma)\otimes\QQ\ar[r]\ar[d]&
 ko_{n+1}(B\Gamma)\otimes\QQ\ar[d]\ar@{^{(}->}[r]^{\mu^\Gamma_{B\Gamma}} \ar[d]& KO_{n+1}(B\Gamma)\otimes\QQ\ar[d]\ar[r]^{\mu^\Gamma_{B\Gamma}}&KO_{n+1}(C^*_{\RR}\Gamma)\otimes\QQ\\
 \Omega^{\spin}_{n+1}(u)\otimes\QQ\ar[r]&ko_{n+1}(u)\otimes\QQ \ar@{^{(}->}[r]&KO_{n+1}(u)\otimes\QQ & }
 \]
 the composition of the horizontal arrows in the first row is zero as a
 consequence of the exactness of \eqref{diagramStolz} and
 \eqref{diagramHR}. As $\mu^\Gamma_{B\Gamma}$ is injective by assumption, it follows that the map $\mathrm{Pos}^{\spin}_{n+1}(B\Gamma)\otimes\QQ\to ko_{n+1}(B\Gamma)$ is zero, too.
 Thus, we get the following factorization of the bottom line
 \begin{equation*}
 \Omega^{\spin}_{n+1}(u)\otimes\QQ\to \mathrm{coker}\big(\mathrm{Pos}^{\spin}_{n+1}(B\Gamma)\to\Omega^{\spin}_{n+1}(u)\big)\otimes\QQ\to ko_{n+1}(u)\otimes\QQ.\end{equation*}
 Since by Lemma~\ref{pontrjagin-surjective} we know that
 $\Omega^{\spin}_{n+1}(u)\otimes\QQ\to ko_{n+1}(u)\otimes\QQ$ is surjective, we have that $\mathrm{coker}\big(\mathrm{Pos}^{\spin}_{n+1}(B\Gamma)\to\Omega^{\spin}_{n+1}(u)\big)\otimes\QQ\to ko_{n+1}(u)\otimes\QQ$ is surjective, too.
 Then, since by the relative Stolz exact sequence \eqref{eq:rel_Stolz} $\mathrm{coker}\big(\mathrm{Pos}^{\spin}_{n+1}(B\Gamma)\to\Omega^{\spin}_{n+1}(u)\big)$ is a subspace of $ \mathrm{Pos}^{\spin}_{n}(X)$, we conclude that
 \[\dim\mathrm{Pos}^{\spin}_{n}(X)\otimes\QQ\geq \dim ko_{n+1}(u)\otimes\QQ= \dim \bigoplus_{j\geq0}H_{n+1-4j}(u;\QQ)=k+k'.\]
\end{Remark}

\section{Concordance classes}\label{4}

The basis of most constructions of positive scalar curvature metrics
is the surgery theorem of Gromov and Lawson, see
\cite{GromovLawson} or \cite{EbertFrenck} for full details. It says
that, given a bordism $W$ from~$M_0$ to~$M_1$ such that~$W$ is obtained from
$M_0$ attaching handles only of codimension $\ge 3$, then a metric of
positive scalar curvature on $M_0$ can be extended to a metric of positive
scalar curvature on $W$ with product structure near the boundary. In
particular, one obtains a ``transported'' positive scalar curvature metric
on $M_1$. We call bordisms satisfying the codimension condition
\emph{Gromov--Lawson admissible}.

In the following, we discuss the details how Gromov--Lawson admissible
bordism $W$ can be obtained, focusing on the not quite so obvious question
why finitely many surgery steps suffice. The result appears also, e.g., as
\cite[Theorem~2.2]{Rosenberg} where the finiteness questions are not
discussed or in a much more general setup in \cite[Appendix~2]{HebestreitJoachim}.

\begin{Proposition}\label{prop:change_bord} Let $X$ be a CW-complex.
 Let $F\colon W\to X$ be a spin bordism with reference map between $f_0\colon
 M_0\to X$ and $f_1\colon M_1\to X$. Assume $f_1$ is
 $2$-connected $(\pi_j(f_1)$ is an isomorphism for $j<2$ and an epimorphism
 for $j=2)$ and
 $n:=\dim(M_1)\ge 5$. Then we
 can change $W$ in the interior to $F'\colon W'\to X$ such that $W'$ is a
 Gromov--Lawson admissible bordism from~$M_0$ to~$M_1$.
\end{Proposition}
\begin{proof}
 By standard results from surgery theory (compare~\cite{Wall}), the desired bordism $W'$ is
 Gromov--Lawson admissible if the inclusion $M_1\into W'$ is $2$-connected. We
 perform surgeries in the interior of $W$ to achieve this.

 \textbf{Isomorphism on $\boldsymbol{\pi_0}$.} We treat each component $A$ of $X$ (or
 equivalently of $M_1$) at a time. We have then to modify $W_A:=F^{-1}(A)$
 so that it becomes connected. This is achieved by (interior) connected sum
 of the finitely many components of $W'$. Because also $X$ is
 path-connected, the map $F\colon W\to X$ can be extended over the
 connected sum of its components.

 \textbf{Isomorphism on $\boldsymbol{\pi_1}$.} The composition $\pi_1(M_1)\to
 \pi_1(W)\to \pi_1(X)$ is an isomorphism, therefore the map $\pi_1(W)\to
 \pi_1(X)$ is surjective. We want to modify $W$ with further
 surgeries which eliminate
 its kernel, then $\pi_1(W')\to\pi_1(X)$ and consequently also $\pi_1(M_1)\to
 \pi_1(W')$ is an isomorphism. As $\pi_1(M_1)\cong \pi_1(X)$ is
 finitely presented (as fundamental group of a smooth compact manifold)
 and $\pi_1(W)$ is finitely generated, this
 kernel is finitely generated as a normal subgroup, see Lemma~\ref{ker-fg} below. So we have to do a finite
 number of surgeries along embedded circles (in the interior of~$W$). Because
 $W$ is oriented, these automatically have trivial normal bundle, so surgery
 is possible. It is also well-known that, possibly after changing the trivialization by the non-trivial element of
 $\pi_1(SO(n))\cong \integers/2$, one can equip the result of the surgeries
 with a spin structure. The fact
 that we do surgery along elements which lie in the kernel of
 $\pi_1(F)$ means precisely that $F$ can be extended over the disks and thus
 over the new bordism, which we continue to denote $W$ by small abuse of
 notation.

 \textbf{Epimorphism on $\boldsymbol{\pi_2}$.} We finally have to perform surgeries so
 that $\iota_*\colon \pi_2(M_1)\to \pi_2(W)$ becomes surjective, where
 $\iota\colon M_1\into W$ is the inclusion. We follow the proof of \cite[Lemma~5.6]{Stolz} adapted to our situation.
 Since $M_1$ and $W$ are compact manifolds, the relative 2-skeleton
 $(W,M_1)^{(2)}$ of $W$ is obtained by attaching a finite number of 2-cells
 to $M_{1}$. To see this one starts with a handlebody decomposition of $W$
 relative to $M_1$ and then uses handle cancellation (the standard results
 from surgery theory, \cite{Wall}, alluded to above) to get rid of $0$-handles and $1$-handles
 using the fact that $M_1\into W$ now is $1$-connected. The $2$-skeleton
 $(W,M_1)^{(2)}$ of
 the CW-decomposition of $W$ relative $M_1$ arising from this new handlebody
 is then homotopy equivalent to $M_{1}\vee\big(\bigvee_{j\in J}
 S^2\big)$.\footnote{Note that this implies a special case of the following
 result of Wall \cite[proof of Lemma~1.1]{Wall_finite}: if $X$, $Y$ are finite connected
 CW-complexes and $f\colon X\to Y$ induces an isomorphism on $\pi_1$,
 then~$\pi_2(f)$, here isomorphic to $\coker(\pi_2(f)\colon \pi_2(X)\to
 \pi_2(Y))$ is finitely generated as a module over the (common)~$\pi_1$.}

 Now, the cokernel of $\iota_*\colon \pi_2(M_1)\to \pi_2(W)$
 is finitely generated by these spheres $x_j$, for $j\in J$. Since $(f_1)_*\colon \pi_2(M_1)\to \pi_2(X)$ is surjective, there exist
 elements $\{y_j\in \pi_2(M_1)\}_{j\in J}$ such that $(f_1)_*(y_j)=F_*(
 x_j)$. It follows that the alternative generators of the cokernel given by $\iota_*\big(y_j^{-1}\big) x_j$ satisfy
 \[F_*\big(\iota_*\big(y_j^{-1}\big) x_j\big)=(\iota\circ
 F)_*\big(y_j^{-1}\big)F_*(x_j)=(f_1)_*\big(y^{-1}_j\big)F_*(x_j)=0\qquad\forall\, j.\]
 These generators we can
 assume embedded because $n\ge 5$. Since $W$ is spin, the normal
 bundle of these embedded spheres is automatically trivial and surgery
 along them is possible.

Because of this, we can extend $F$ over the surgeries along
 the alternative generators $\iota_*\big(y_j^{-1}\big) x_j$ and we obtain
 the desired cobordism ${F'}\colon W'\to X$ such that the
 inclusion of $M_1$ into $W'$ is a~$2$-equivalence.

 We conclude by explaining why this now is a Gromov--Lawson admissible
 bordism. Start with an arbitrary handle decomposition of~$W$ relative to~$M_1$.
 Now we are in the situation to apply the handle cancellation method~\cite{Wall} again: because the map is $2$-connected, we find an alternative
 handle decomposition without $0$-, $1$-, or $2$-handles. Turning this
 upside-down this handle decomposition can be interpreted as a handle
 decomposition of $W$ relative to $M_0$. In this interpretation, what was
 previously the dimension of the handle now becomes the codimension. As the
 result, we obtain $W$ from $M_0$ attaching handles only of codimension $\ge
 3$, as desired.
\end{proof}

\begin{Lemma}\label{ker-fg}
 Let $\alpha\colon \Gamma'\to\Gamma$ be a surjective group homomorphism
 between finitely generated groups. Assume in addition that $\Gamma$ is
 finitely presented. Then the kernel of $\alpha$ is finitely generated as a
 \emph{normal} subgroup of $\Gamma'$.
\end{Lemma}
\begin{proof}
 Let us fix a finite presentation $ \Gamma=\langle x_1,\dots, x_h ; r_1, \dots, r_k\rangle $, where the relations $r_j$ are given by fixed words $w_j\big(x_1^{\pm1},\dots, x_h^{\pm1}\big)$. Let us fix also a finite set of generators $\{ y_1,\dots, y_n\}$ for~$\Gamma'$.
 Pick $a_1,\dots,a_h\in\Gamma'$ such that $\alpha(a_j)=x_j$ for all $j$ and set
 $w_l'\big(x_1^{\pm1},\dots, x_h^{\pm1}\big):=\alpha(y_l)$. Then it follows
 that \begin{gather*}\big\{w_1\big(a_1^{\pm1},\dots, a_h^{\pm1}\big),\dots, w_k\big(x_1^{\pm1},\dots,
 x_h^{\pm1}\big),y^{-1}_1w_1'\big(x_1^{\pm1},\dots, x_h^{\pm1}\big), \dots,
 y^{-1}_nw_n'\big(x_1^{\pm1},\dots, x_h^{\pm1}\big)\!\big\}\end{gather*}
 is a finite set of generators
 as a normal subgroup for $\ker \alpha$.
\end{proof}

Now we are ready for the proof of the main result of this section.

\begin{proof}[Proof of Theorem \ref{rank>k}]
 Let us consider again the situation of Theorem \ref{lower-bound}, where we have
 classes $x_1=\big[M_1\xrightarrow{f_1}X,h_1\big]$,
 $\dots$, $x_k=\big[M_k\xrightarrow{f_k}X,h_k\big]$ in $\Pos^{\spin}_{n}(X)$ which span a subgroup of rank $k$,
 but are trivial when mapped to $\Omega^{\spin}_n(X)$ (and a fortiori to $\Omega^{\spin}_n(B\Gamma)$), in particular they are null-bordant.
 Let us pick such null-bordisms $F^i\colon Y_i\to X$, so that $M_i$ is the boundary of $Y_i$ and $f_i$ is the restriction of $F^i$ to the boundary.

 For $i\in\{1,\dots,k\}$, the disjoint union of $M$ and $M_i$ is spin bordant to $M$, with bordism $ G_i\colon W_i\to X$
 given by the disjoint union of $f\times id\colon M\times [0,1]\to X$ and
 $F_i\colon Y_i\to X$. By Proposition~\ref{prop:change_bord} we can modify
 these bordisms in the interior and then assume that
 $W_i$ is Gromov--Lawson admissible.

Now we can use the Gromov--Lawson surgery theorem
 to ``push'' the given metrics $g\disjointunion h_i$ of positive scalar curvature
 from $M\disjointunion M_i$ through the new
 bordism to positive scalar curvature metrics $g_i$ on $M$. We obtain new
 representatives $\big[M\xrightarrow{f} X, g_i\big]$ of $x_i$. As the $x_i$ span an
 affine lattice of rank $k$ in $\Pos^{\spin}(X)$, this finishes the
 proof.
\end{proof}

Let us spell out the special case $X=M$ of Theorem \ref{rank>k}:
\begin{Corollary}\label{corol:metrics}
 Let $(M,g)$ be an $n$-dimensional connected
 spin manifold
 of positive scalar curvature with $\pi_1(M)=\Gamma$, $n\ge 5$. Let $u\colon M\to
 B\Gamma$ be an isomorphism on fundamental groups. Set
 \begin{gather*}
 k:=\sum_{ \substack{0\le j\le n+1\\ j-n\equiv 1\pmod 4}} \dim\left(\coker(u_*\colon
 H_j(M;\rationals)\to H_j(B\Gamma;\rationals))\right).
 \end{gather*}
 Then $M$ admits metrics $g_1,\dots,g_k$ of positive scalar curvature such
 that the elements $(M,g)$, $(M,g_1), \dots, (M,g_k)$ span a $k$-dimensional affine
 subspace of $\Pos^{\spin}_n(M)\tensor\rationals$. In particular, these
 metrics form a $k$-dimensional lattice of non-concordant metrics of
 positive scalar curvature on~$M$.
\end{Corollary}

\begin{Example}
 Let $(M,g)$ be a connected $n$-dimensional spin manifold of
 positive scalar curvature such that
 $\dim(H_{n+1}(B\pi_1(M);\rationals))=k$. Then
 the cokernel of the map induced by the inclusion in homology of degree
 $n+1$ is $H_{n+1}(B\pi_1(M))$, as $H_{n+1}(M)=0$ for degree
 reasons. Therefore there exist $k$ metrics $g_1,\dots,g_k$ of
 positive scalar curvature on $M$ which span
 together with $g$ an affine space of rank $k$ in
 $\Pos^{\spin}_n(M)$ and, in particular, give rise to a lattice of
 rank $k$ of concordance classes of positive scalar curvature
 metrics on $M$.

For example, if $\pi_1(M)\iso\integers^N$ then we have
 $\dim\big(H_{n+1}\big(\integers^N;\rationals\big)\big)=\binom{N}{n+1}$.
\end{Example}

\begin{Example}\label{ex:given_pi1}
 Assume that $n\ge 5$ and $\Gamma=\langle x_1\dots,x_k\,|\,
 r_1,\dots,r_h\rangle$ is finitely presented. Then there
 exists a closed spin manifold~$M$ of dimension $n$ with fundamental group
 $\Gamma$ which admits a metric~$g$ of positive scalar curvature.

 Indeed, take the wedge of $k$ circles and, for each relation
 $r_i$, attach a two cell. Denote by $X$ this 2-dimensional CW-complex. Finally
 embed $X$ into $\RR^{n+1}$ and consider a
 tubular neighbour\-hood~$\mathcal{N}$ of~$X$. Then
 $M:=\partial\mathcal{N}$ is an $n$ dimensional spin manifold
 with fundamental group $\Gamma$. Observe that $\mathcal{N}$ is
 a spin $B\Gamma$ null-bordism for $M$ or, after cutting out a disk, a
 $B\Gamma$ bordism from~$S^n$ to~$M$. By Proposition~\ref{prop:change_bord} we can
 assume that this bordism is Gromov--Lawson admissible. Then~$M$ admits a metric
 $g$ of positive scalar curvature by the Gromov--Lawson
 surgery theorem.

 In addition, observe that for the manifold we constructed we have a
 factorization $u\colon M\to \mathcal{N}\to B\Gamma$ and $\mathcal{N}$ is
 homotopy equivalent to a $2$-dimensional CW-complex. Therefore
 \begin{equation*}
 \im(u_*\colon H_j(M)\to H_j(B\Gamma))=\{0\}\qquad\forall\, j\ge 3.
 \end{equation*}

 By Corollary~\ref{corol:metrics}, with $k= \sum_{3\le j\le n+1,\, j\equiv
 n+1\pmod 4} \dim(H_j(B\Gamma;\QQ))$ we find metrics $g_1,\dots,g_k$ of
 positive scalar curvature on $M$ such that the $(M,g_i)$ together with $(M,g)$
 span a $k$-dimensional affine subspace of
 $\Pos^{\spin}_n(M)\tensor\rationals$. Note that this includes many examples
 where $k=+\infty$, whenever the rational homology of $\Gamma$ is not finitely
 generated in the appropriate degrees.
\end{Example}

\begin{Remark}In Example \ref{ex:given_pi1} we focused on $\Pos^{\spin}_n(M)\tensor\rationals$. Nevertheless, as one of the referees suggested, if we consider $P^+(M)$ we can obtain the better estimate of its rank as follows:
 \begin{equation*}
 \dim P^+(M)\otimes\QQ= \dim \mathrm{R}^{\spin}_n(M)\tensor\rationals\geq
 \dim \big(\im(\Omega^{\spin}_{n+1}(B\Gamma)\otimes\QQ\to KO_{n+1}(B\Gamma)\otimes\QQ)\big),
 \end{equation*}
 where the last inequality is obtained by combining the mapping
 from the Stolz exact sequence to the Higson--Roe exact
 sequence and the assumption that~$\Gamma$ verifies the
 (rational) strong Novikov conjecture. Then we deduce from to Lemma~\ref{pontrjagin-surjective} that
 \begin{equation}\label{eq:bigger_P_bound}
 \dim P^+(M)\otimes\QQ\geq \sum_{j\geq0}\dim H_{n+1-4j}(B\Gamma;\QQ).\end{equation}
 The difference between this lower bound and the one obtained
 in Theorem~\ref{rank>k} is due to the fact that there
 can be metrics which are not concordant (hence different in
 $P^+(M)$) but bordant (hence equal in $\Pos^{\spin}_n(M)$).
 Indeed, to understand where the larger number of linearly
 independent metrics predicted by~\eqref{eq:bigger_P_bound}
 compared to Theorem \ref{rank>k} come from, let us
 consider \[\im\bigg(\bigoplus_{j\geq0}H_{n+1-4j}(M;\QQ)\xrightarrow{u_*}
 \bigoplus_{j\geq0}H_{n+1-4j}(B\Gamma;\QQ)\bigg).\] Let $V'\subset
 \bigoplus_{j\geq0}H_{n+1-4j}(M;\QQ)$ be a subspace of maximal
 dimension on which $u_*$ is injective. By Lemma~\ref{pontrjagin-surjective} we can lift it to a subspace $V$ of
 $\Omega_n^{\spin}(M)\tensor\rationals$ and, because of the
 assumption about the (rational)
strong Novikov conjecture, $V$ injects into
 $\mathrm{R}^{\spin}_n(M)\otimes\QQ$ and provides an additional
 subspace of $P^+(M)\otimes\QQ$. It is immediate to see that, by
 the exactness of the Stolz sequence, $V$ is mapped to $0$ in $\mathrm{Pos}^{\spin}_n(M)\otimes\QQ$
 and therefore all the metrics given by~$V$ in $R^+(M)$ are
 null-bordant.
 We are going to see examples of this kind of metrics in Remark~\ref{carr}.
\end{Remark}

\begin{Remark}
 In special situations, the different metrics constructed in Theorem~\ref{rank>k}, Corollary~\ref{corol:metrics} and the examples remain
 different also in the moduli space of Riemannian metrics of positive scalar
 curvature on $M$, the quotient by the action of the diffeomorphisms
 group. This is worked out in detail in~\cite{PiazzaSchickZenobi_pairing}. As indicated in the introduction, this is based
 on the use of higher numeric rho invariants, whose behavior under the action
 of the diffeomorphism group can be controlled.
\end{Remark}

\begin{Remark}Consider the map $i$ from \eqref{i-map}. If we compose it
 with the map ${\Ind}^{\Gamma}$ in $\eqref{HRses}$, it is easy to see
 that we obtain the map \[\mathrm{Ind\, diff}^\Gamma\colon \ P^+(M)\times
 P^+(M)\to KO_{n+1}(C^*\Gamma)\] used in \cite[Section~5.3]{EbertRW}. More
 precisely, in \cite{EbertRW} the map is defined on the space of isotopy
 classes of metrics with positive scalar curvature, but it descends to
 $P^+(M)$, using that the different definitions of $\mathrm{Ind\,diff}^\Gamma$ all
 coincide, proved in detail in~\cite{Ebert_inddiff} and for non-trivial~$\Gamma$ in the M\"unster dissertation of Buggisch~\cite{Buggisch}.

 It is straightforward to see that, rationally, the affine subspace generated by the lattice of~$P^+(M)$ in Theorem~\ref{rank>k}, $(3)$ is mapped surjectively onto
 the image of the rational assembly map $ko_{n+1}(B\Gamma)\otimes\QQ\to KO_{n+1}(C^*\Gamma)\otimes\QQ$.
\end{Remark}

\begin{Remark}\label{carr}
 As a predecessor construction of concordance classes which does not make use of non-trivial torsion, let us recall the construction of Carr, see~\cite{Carr}.

 First, consider the sphere~$S^{4n-1}$. Carr takes a 2-connected $4n$ dimensional spin manifold~$B$ with $\hat{A}(B)=1$ and removes two disks to obtain a bordism $W$ from $S^{4n-1}$ to $S^{4n-1}$. Positive scalar curvature surgery produces a metric of positive scalar curvature on $W$ starting with the canonical metric on $S^{4n-1}$ and ending with a non-concordant new metric of positive scalar curvature on $S^{4n-1}$.

 However, these metrics are equal in $\Pos^{\spin}_{4n-1}\big(S^{4n-1}\big)$. To see this, we have just to construct the reference map $F\colon W\to S^{4n-1}$ which restricts to the identity on the boundary components.
 For this, choose a path which is a clean embedding of the closed interval into~$W$, joining two points in the two boundary spheres. Choose then a tubular neighbourhood of this one dimensional submanifold of~$W$, which is necessarily trivial. Now a trivialization of the tubular neighbourhood defines a collapse map from $W$ to $S^{4n-1}$, whose restriction to the boundary components is homotopic to the identity. Putting these homotopies on collar neighbourhoods of the boundary components, we obtain the desired map~$F$.

 More generally, given an arbitrary closed spin manifold $M$ of dimension $4n-1$ with positive scalar curvature metric $g$, Carr makes a connected sum of $M\times [0,1]$ with $W$ along a path parallel to the previously chosen one, to obtain a psc bordism $V$ from $(M,g)$ to $(M,g')$. These two metrics have non-zero index difference and therefore they are not concordant.
 Nevertheless, they are equal in $\Pos^{\spin}_{4n-1}(M)$. We obtain the desired reference map from $V $ to $M$ by connected sum of the previous map with the projection from $M\times [0,1]$ to $M$.
\end{Remark}

\subsection*{Acknowledgements}
The authors thank the German Science Foundation and its priority program ``Geometry at Infinity'' for partial support. We thank the
referees for a number of helpful suggestions improving the presentation and helping to avoid inaccuracies.

\pdfbookmark[1]{References}{ref}
\LastPageEnding

\end{document}